\documentclass[12pt,a4]{amsart}
\usepackage{a4wide}
\usepackage{amsmath}
\usepackage{amsfonts}
\usepackage{amsthm}

\theoremstyle{plain}
\newtheorem{thm}{Theorem}[section]
\newtheorem{cor}[thm]{Corollary}
\newtheorem{lem}[thm]{Lemma}
\newtheorem{prop}[thm]{Proposition}
\theoremstyle{definition}

\renewcommand{\ker}{\operatorname{ker}}

\newcommand{\im}{\operatorname{im}}

\newcommand{\rank}{\operatorname{rank}}

\newcommand{\en}{\operatorname{End}}

\newcommand{\ig}{\operatorname{IG}}

\newcommand{\el}{\mbox{$\mathcal{L}$}}
\newcommand{\dee}{\mbox{$\mathcal{D}$}}

\newcommand{\ar}{\mbox{$\mathcal R$}}

\newcommand{\leqel}{\mbox{$\leq _{\mathcal L }$}}
\newcommand{\leqar}{\mbox{$\leq _{\mathcal R}$}}

\begin{document}

\title[Maximal subgroups of IG$(E)$]{Every group is the maximal subgroup of a naturally occurring free idempotent generated semigroup}
\thanks{\phantom{au} Research supported by EPSRC grant no.
EP/I032312/1.The authors would like to thank Robert Gray and Nik Ruskuc for some useful discussions.}
\date{\today}

\author{Victoria Gould} 
\email{victoria.gould@york.ac.uk}
\author{Dandan Yang}
\email{ddy501@york.ac.uk}
\address{Department of Mathematics\\University
  of York\\Heslington\\York YO10 5DD\\UK}

\begin{abstract} Gray and Ruskuc have shown that any group $G$ occurs as the maximal subgroup of some free idempotent generated semigroup IG$(E)$ on a biordered set of idempotents $E$, thus  resolving a long standing open question.  Given the group $G$, they  make a careful choice for $E$ and use a certain amount of well developed machinery. Our 
aim here is to present a short and direct proof of the same result, moreover by using a naturally occuring biordered set. 

More specifically, for any free $G$-act $F_n(G)$ of
finite rank $n\geq 3$, we have that $G$ is a maximal subgroup of IG$(E)$ where $E$ is the biordered set of idempotents of End $F_n(G)$. Note that if $G$ is finite then so is 
End $F_n(G)$.
\end{abstract}

\maketitle

\section{Introduction} 

Let $S$ be a semigroup   and denote by
$\langle E(S)\rangle$ the subsemigroup of $S$ generated by the set of idempotents
$E(S)$ of $S$. If $S=\langle E(S)\rangle$, then we say that $S$ is {\em idempotent generated}.  In a landmark paper, Howie \cite{howie:1966} showed that every semigroup may be embedded into one that is idempotent generated, thus making transparent the importance of the role played by such semigroups. In the same article, Howie showed that the semigroup of non-bijective endomorphisms
of a finite set to itself is idempotent generated.  This latter  theorem was quickly followed by a `linearised' version due to Erd\"{o}s \cite{erdos:1967}, who proved that the
multiplicative semigroup of singular square matrices over a field is idempotent generated.  Fountain and Lewin \cite{fountainandlewin} subsumed these results into the wider context of endomorphism monoids of independence algbras. Indeed Howie's work can be extended in many further ways: see, for example, \cite{basisiii,putcha:2006}. We note here that sets and vector spaces over division rings are examples of independence algebras, as are free (left) $G$-acts over a group $G$. 

For any set of idempotents $E=E(S)$ there is a free object IG$(E)$ in the category of semigroups that are generated by $E$, given by the presentation
\[\ig(E)=\langle \overline{E}:  \bar{e}\bar{f}=\overline{ef},\, e,f\in E, \{ e,f\}
\cap \{ ef,fe\}\neq \emptyset\rangle,\]
where here $\overline{E}=\{ \bar{e}:e\in E\}$.\footnote{ It is more usual to identify elements of
$E$ with those of $\overline{E}$, but it helps the clarity of our later arguments to
make this distinction.} We say that $\ig(E)$ is the {\em free idempotent generated semigroup over $E$}.  The relations in the presentation for $\ig(E)$ correspond to taking {\em basic products} in $E$, that is, products between $e,f\in E$ where
$e$ and $f$ are comparable under one of the quasi-orders $\leqel$ or $\leqar$ defined on $S$. In fact, $E$ has an abstract characterisation as a {\em biordered set}, that is, a partial algebra equipped with two quasi-orders satisfying certain axioms. A celebrated result of Easdown \cite{easdown:1985} shows every biordered set $E$ occurs as $E(S)$ for some semigroup $S$, hence we lose nothing by assuming that our set of idempotents is of the form $E(S)$ for a semigroup $S$. 

The semigroup $\ig(E)$ has some pleasant properties. It follows from the definitions that the natural map
$\phi:\ig(E)\rightarrow S$, given by $\bar{e}\phi= e$, is a morphism onto $\langle E(S)\rangle$. Since any morphism preserves $\el$-classes and $\ar$-classes, 
certainly so does $\phi$. Consequently, $\phi$ is a morphism from any maximal subgroup
$H_{\bar{e}}$ of $\ig(E)$ onto $H_e$. More remarkably, we have the following lemma, taken from \cite{fitzgerald:1972,nambooripad:1979,easdown:1985}.

\begin{prop}\label{prop:remarkable} Let $E=E(S)$, the free idempotent generated semigroup $ \ig(E)$ and $\phi$ be as 
above. 
\begin{enumerate}\item[(i)] The restriction of $\phi$ to the set of idempotents of $\ig(E)$ is a bijection
onto E (and an isomorphism of biordered sets).
\item[(ii)] The morphism $\phi$  induces a bijection between the set of
all $\ar$-classes (respectively $\el$-classes) in the $\mathcal{D}$-class of $\bar{e}$ in $\ig(E)$ and the
corresponding set in $\langle E(S)\rangle$.
\end{enumerate}
\end{prop}

Biordered sets were introduced by Nambooripad \cite{nambooripad:1979} in his seminal work on the structure of regular semigroups, as was the notion of free idempotent generated semigroups $\ig(E)$. A significant conjecture, which although being of longstanding appears not to have appeared formally until 2002
\cite{mcelwee:2002}, purported that all maximal subgroups of $\ig(E)$ were free. This conjecture was disproved by Brittenham, Margolis and Meakin \cite{brittenham:2009}. 
The result motivating our current paper is the main theorem of \cite{gray:2012}, in which Gray and Ruskuc show that {\em any} group occurs as the maximal subgroup of some
$\ig(E)$. Their proof involves machinery developed by Ruskuc to handle presentations of maximal subgroups, and, given a group $G$, a careful construction of $E$. 

Of course the question remains of whether a group $G$ occurs as a maximal subgroup of some
 $\ig(E)$ for a `naturally occuring' $E$. The signs for this were positive, given recent work in \cite{brittenham:2011} and \cite{gray:2012a} showing (respectively) that the multiplicative group of non-zero elements of any division ring $Q$ occurs as a maximal subgroup of a rank 1 idempotent in $\ig(E)$, where $E$ is the biordered set of idempotents of
$M_n(Q)$ for $n\geq 3$, and that any symmetric group $\mathcal{S}_r$ occurs as a maximal subgroup of a rank $r$ idempotent in $\ig(F)$, where $F$ is the biordered set of idempotents of 
a full transformation monoid
$\mathcal{T}_n$ for some $n\geq r+2$. Note that in both these cases, $H_{\bar{e}}\cong
H_e$ for the idempotent in question. 

As pointed out above, sets and vector spaces are examples of {\em independence algebras}, as is any rank $n$ free (left) $G$-act $F_n(G)$. Elements of
$\en F_n(G)$ are endowed with rank (simply the size of a basis of the image) and it is known that the maximal subgroups of rank 1 idempotents are isomorphic to $G$. 
We elaborate on the structure of $\en F_n(G)$ in Section~\ref{sec:enF}, but stress that much that we write can be extracted from known results for independence algebras. 
Once these preliminaries are over,  Section~\ref{sec:theresult} demonstrates in a very direct manner (without appealing to presentations) that for $E=E(\en F_n(G))$ and 
$e\in \en F_n(G)$ with $n\geq 3$ and $\rank e=1$ we have $H_{\bar{e}}\cong H_e$, thus showing that any group occurs as a maximal subgroup of some {\em natural} $\ig(E)$. 
However, although none of the technicalities involving presentations appear here explicitly, we nevertheless have made use of the essence of some of the arguments
of \cite{dolinka:2012,gray:2012}, and more particularly earlier observations of Ruskuc \cite{ruskuc:1999}
concerning sets of generators for subgroups.

\section{Preliminaries: free $G$-acts, rank-1 $\mathcal{D}$-classes and singular squares}\label{sec:enF}

Let $G$ be a group and let $F_n(G)=\bigcup_{i=1}^n Gx_{i}$ be a
rank $n$ free left $G$-act with $n\in\mathbb{N},n\geq 3$. We recall that, as a set,
$F_n(G)$ consists of the set of formal symbols $\{ gx_i:g\in G, i\in [1,n]\}$,
where $[1,n]=\{ 1,\hdots,n\}$. For any $g,h\in G$ and $1\leq i,j\leq n$ we have
that $gx_i=hx_j$ if and only if $g=h$ and $i=j$; the action of $G$ is given by
$g(hx_i)=(gh)x_i$, and we usually identify $x_i$ with $1x_i$, where $1$ is the identity of $G$. For our main result, it is enough to take $n=3$, but for the sake of generality we proceed with arbitrary $n\geq 3$. Let End $F_n(G)$ denote the endomorphism monoid of $F_n(G)$ (with composition left-to-right). The image of  
$a\in\en F_n(G) $ being a $G$-subact, we can define the 
{\em rank} of $a $ to be the rank of $\im a$. 

Since $F_n(G)$ is an independence algebra, a direct application of Corollary 4.6 \cite{gould:1995} gives  a useful characterization of Green's relations on End $F_n(G)$.

\begin{lem}\cite{gould:1995} \label{lem:green}
For any $a, b\in \en F_n(G)$, we have the following:

\begin{enumerate}\item[(i)] $\im a =\im b$ if and only if $a\,\el\, b$;

\item[(ii)] $\ker a=\ker b$ if and only if $a\,\ar\, b$;

\item[(iii)] $\rank a= \rank b$ if and only if $a\,\dee\, b$.
\end{enumerate}
\end{lem}

 Clearly  elements
$a,b\in \en F_n(G)$ depend only on 
their action on the free generators $\{ x_i:i\in [1,n]\}$
and it is therefore convenient to write \[a=\begin{pmatrix}
x_1&\hdots &x_n\\
a_1x_{1\overline{a}}&\hdots &a_nx_{n\overline{a}}\end{pmatrix}
\mbox{ and }b=\begin{pmatrix}
x_1&\hdots &x_n\\
b_1x_{1\overline{b}}&\hdots& b_nx_{n\overline{b}}\end{pmatrix}.\] Let $D=\{a\in \en F_n(G)\mid  \rank a=1\}$. Clearly 
$a,b\in D$ if and only if $\overline{a},\overline{b}$ are constant, and from Lemma~\ref{lem:green} we have $a\,\el\, b$ if and only if 
$\im\overline{a}=\im\overline{b}$.

\begin{lem}\label{lem:1kernels}
Let $a, b\in D$ be as above. Then $\ker a=\ker b$ if and only if
$(a_1,\hdots ,a_n)g=$ $(b_1,\hdots,b_n)$ for some $g\in G$. 
\end{lem}

\begin{proof}
Suppose $\ker a = \ker b$. For any $i,j\in [1,n]$
we have  $(a_{i}^{-1}x_{i})a=x_{i\overline{a}}=(a_j^{-1}x_j)a$ so that by assumption,
$(a_{i}^{-1}x_{i})b=(a_j^{-1}x_j)b$.
Consequently, 
$a_i^{-1}b_i=a_j^{-1}b_j=g\in G$ and it follows that
$(a_1,\hdots ,a_n)g=(b_1,\hdots,b_n)$.

Conversely, if $g\in G$ exists as given then for any $u,v\in G$ and $i,j\in [1,n]$ we have  $$(ux_{i})a=(vx_{j})a \Leftrightarrow ua_{i}=va_{j}\Leftrightarrow ua_{i}g=va_{j}g\Leftrightarrow ub_{i}=vb_{j}\Leftrightarrow (ux_{i})b=(vx_{j})b.$$
\end{proof}

We index the $\el$-classes in $D$ by $J=[1,n]$, where
the image of $a\in L_j$ is $Gx_j$, and we  index the  $\mathcal{R}$-classes of $D$ by  $I$, so that by Lemma~\ref{lem:1kernels}, the set $I$ is in bijective correspondence with $G^{n-1}$. From \cite[Theorem 4.9]{gould:1995} we have that $D$ is a completely simple semigroup.   We use $e_{ij}$ to denote the identity of the $\mathcal{H}$-class $H_{ij}$. For convenience we also suppose that
$1\in I$ and let 
\[e_{11}=\begin{pmatrix}
x_{1} & \cdots & x_{n} \\
x_{1} & \cdots & x_{1} \end{pmatrix}.\] Clearly, for any given $i\in I,j\in J$ we have
\[e_{1j}=\begin{pmatrix}
x_{1} & \cdots & x_{n} \\
x_{j} & \cdots & x_{j} \end{pmatrix}\mbox{ and }e_{i1}=\begin{pmatrix}
x_{1} & x_{2} & \cdots & x_{n} \\
x_{1} & a_{2i}x_{1} & \cdots & a_{ni}x_{1} \end{pmatrix},\] where $a_{2i},\cdots,a_{ni}\in G$. 

\begin{lem}\label{lem:Goccurs}
Every $\mathcal{H}$-class of $D$ is isomorphic to $G$.
\end{lem}

\begin{proof}
By standard semigroup theory, we know that any two group $\mathcal{H}$-classes in the same $\mathcal{D}$-class are isomorphic, so we need only show that $H_{11}$ is isomorphic to $G$. By Lemma~\ref{lem:1kernels} an element $a\in\en F_n(G)$ lies in $ H_{11}$ if and only if $a=a_g=\begin{pmatrix}
x_{1} & \cdots & x_{n} \\
gx_{1} & \cdots & gx_{1} \end{pmatrix}$, for some $g\in G$. It is easy to check that $\psi: H\rightarrow G$ defined by  $a_g\psi= g$ is an isomorphism. 
\end{proof}

Since $D$ is a completely simple semigroup, it is isomorphic to some Rees matrix semigroup $\mathcal{M}=\mathcal{M}(H_{11}; I, J; P)$, where 
 $P=(p_{j i})=(q_{j}r_{i})$, and we can take $q_{j}=e_{1j}\in H_{1j}$ and $r_{i}=e_{i1}\in H_{i1}$. Since the $q_j,r_i$ are chosen to be idempotents, it is clear that
$p_{1i},p_{j1}=e_{11}$ for all $i\in I, j\in J$. 

\begin{lem}\label{lem:thecolumns} For any $a_{g_2},\hdots ,a_{g_n}\in H_{11}$,  we can choose $k\in I$ such that the $k$th column of $P$ is  $(e_{11},a_{g_2},\hdots ,a_{g_n})^T$. 
\end{lem}
\begin{proof} Choose $k\in I$ such that $e_{k1}=
\begin{pmatrix} x_1&x_2&\hdots &x_n\\ x_1&g_2x_1&\hdots &g_nx_1\end{pmatrix}$ (note that if $g_2=\hdots =g_n=1$ then $k=1$).
\end{proof}

Let $E$ be a biordered set; from \cite{easdown:1985} we can assume that
$E=E(S)$ for some semigroup $S$. An {\it $E$-square} is a sequence $(e,f,g,h,e)$ of elements of $E$ with $e~\mathcal{R}$$~f~\mathcal{L}$$~g~\mathcal{R}$$~h~\mathcal{L}$$~e$. We draw such an $E$-square as $\begin{bmatrix} e&f\\ h&g\end{bmatrix}$. 

\begin{lem} \label{lem:esquarerectband} The elements of 
an $E$-square $\begin{bmatrix} e&f\\ h&g\end{bmatrix}$  form a rectangular band 
(within $S$) if and only if one (equivalently, all) of the following four equalities holds: $eg=f$, $ge=h$, $fh=e$ or $hf=g$.
\end{lem}
\begin{proof}
The necessity is clear. To prove the sufficiency, without loss of generality, suppose that the equality $eg=f$ holds. We need to prove $ge=h$, $fh=e$ and $hf=g$. Notice that $gege=gfe=ge$, so $ge$ is idempotent. But, as $f\in L_{g}\cap R_{e}$
\cite[Proposition 2.3.7]{howie:1995} gives that $ge\in R_{g}\cap L_{e}$, which implies $ge=h$. Furthermore, $fh=fge=fe=e$ and $hf=heg=hg=g$, and so
$\{ e,f,g,h\}$  is a rectangular band.
\end{proof}

We will be interested in rectangular bands in completely simple semigroups. The following lemma makes explicit ideas used implicitly elsewhere. We remark that the notation for idempotents used in the lemma fits exactly with that above.

\begin{lem}\label{lem:sandwich} 
Let $\mathcal{M}=\mathcal{M}(G;I,J; P)$ be a Rees matrix semigroup over a group $G$
with sandwich matrix $P=(p_{ji})$. For any $i\in I, j\in J$ write $e_{ij}$ for the idempotent $(i,p_{ji}^{-1},j)$. Then an $E$-square $\begin{bmatrix} e_{ij}&e_{il}\\ e_{kj}&e_{kl}\end{bmatrix}$  is a rectangular band if and only if $p_{ji}^{-1}p_{jk}=p_{li}^{-1}p_{lk}$.
\end{lem}
\begin{proof} We have
\[e_{ij}e_{kl}=e_{il}\Leftrightarrow
(i,p_{ji}^{-1}, j)(k,p_{lk}^{-1},l)=(i,p_{li}^{-1},l)
\Leftrightarrow 
 p_{ji}^{-1}p_{jk}p_{lk}^{-1}=p_{li}^{-1}\Leftrightarrow p_{ji}^{-1}p_{jk}=p_{li}^{-1}p_{lk}.\]
 The result now follows from Lemma~\ref{lem:esquarerectband}. 
\end{proof}

An $E$-square $(e,f,g,h,e)$ is {\it singular} if, in addition, there exists $k\in E$ such that either:
\[
\left\{\begin{array}{ll}
ek=e,\, fk=f, \,ke=h,\,  kf=g \mbox{ or}\\
ke=e,\, kh=h,\, ek=f,\, hk=g.
\end{array}\right.
\]
We call a singular square for which the first condition holds an {\it up-down  singular square}, and that satisfying  the second condition a {\it left-right singular square}.

\section{Free idempotent generated semigroups}\label{sec:theresult}

Continuing the notation of the previous section, the rest of this paper 
is dedicated to prove that the maximal subgroup ${H}_{\overline{e_{11}}}$ is isomorphic to $H_{e_{11}}$ and hence by Lemma~\ref{lem:Goccurs} to $G$. For ease of notation we denote ${H}_{\overline{e_{11}}}$ by $\overline{H}$ and $H_e$ by $H$.

As remarked earlier, although we do not directly use the presentations for maximal subgroups of semigroups developed 
in \cite{ruskuc:1999} and adjusted and implemented for free idempotent generated semigroups in \cite{gray:2012}, we are nevertheless making use of ideas from those papers. In fact, our work may be considered as a simplification of previous approaches,
in particular \cite{dolinka:2012}, in the happy situation where a $\dee$-class is completely simple, our sandwich matrix has the property of Lemma~\ref{lem:thecolumns}, and the next lemma holds.

\begin{lem}\label{lem:singsquaresD}
An $E$-square $\begin{bmatrix} e&f\\ h&g\end{bmatrix}$ in $D$ is singular if and only
$\{ e,f,g,h\}$ is a rectangular band.
\end{lem}

\begin{proof} Suppose that
$\begin{bmatrix} e&f\\ h&g\end{bmatrix}$ is singular. If $k=k^2\in \en F_n(G)$ is such that $ek=e,\, fk=f, \,ke=h$ and   $kf=g$ then $eg=ekf=ef=f$. By Lemma~\ref{lem:esquarerectband}, 
$\{ e,f,g,h\}$ is a rectangular band. Dually for a left-right singular square. 

Conversely, suppose that  
$\{ e,f,g,h\}$ is a rectangular band. If $e\,\el\, f$, then
our $E$-square becomes  $\begin{bmatrix} e&e\\ g&g\end{bmatrix}$ and taking
$k=g$ we see this is an up-down singular square. 

 Without loss of generality we therefore suppose that 
$\{ 1\}=  \im\overline{e}=\im\overline{h}\neq\im\overline{f}=\im\overline{g}=\{ 2\}$. Following standard notation we write
$x_{i}e=e_{i}x_{1}$, $x_{i}h=h_{i}x_{1}$, $x_{i}f=f_{i}x_{2}$ and $x_{i}g=g_{i}x_{2}$. As $e$, $f$, $h$ and $g$ are idempotents, it is clear that $e_{1}=h_{1}=1$ and $f_{2}=g_{2}=1$.  Since 
 $\{ e,f,g,h\}$  is a rectangular band, we have $eg=f$ and so
 $x_{1}eg=x_{1}f$, that is, $g_{1}=f_{1}$. Similarly, from $ge=h$, we have $e_{2}=h_{2}$. Now we define $k\in \en F_n(G)$ by
 \[ x_{i}k = \left\{ \begin{array}{lll}
         x_{1} & \mbox{if $i=1$};\\
         x_{2} & \mbox{if $i=2$};\\
        g_{i}x_{2} & \mbox{else}.\end{array} \right. \]
Clearly  $k$ is idempotent and since $\im e$ and $\im f$ are contained in $\im k$
we have $ek=e$ and $fk=f$. Next we prove that $ke=h$. Obviously, $x_{1}ke=x_{1}h$ and $x_{2}ke=x_{2}h$ from $e_{2}=h_{2}$ obtained above. For other $i\in [1,n]$, 
we use the fact that from Lemma~\ref{lem:1kernels}, there is an $s\in G$ with $h_{i}=g_{i}s$, for all $i\in [1,n]$. Since
$$h_{i}e_{2}^{-1}=g_{i}sh_{2}^{-1}=(g_{i}s)(g_{2}s)^{-1}=g_{i}g_{2}^{-1}=g_{i}$$
we have  $x_{i}ke=(g_{i}x_2)e=g_ie_2x_{1}=h_{i}x_{1}=x_{i}h$ so that $ke=h$. It remains to show that $kf=g$. First, $x_{1}kf=f_{1}x_{2}=g_{1}x_{2}=x_{1}g$ and $x_{2}kf=x_{2}=x_{2}g$. For other $i$, since $x_2f=x_2$  the definition of $k$ gives $x_{i}kf=x_{i}g$. Hence $kf=g$ as required. Thus, by definition,
$\begin{bmatrix} e&f\\ h&g\end{bmatrix}$ is a singular square.\end{proof}

Notice that  the argument above proves the following:

\begin{cor}\label{cor:updown}
An $E$-square in $D$ is a singular square if and only if it is an up-down singular square.
\end{cor}

\begin{lem}\label{lem:easy} 
For any idempotents $e,f,g\in D$, $ef=g$ implies $\overline{e}~\overline{f}=\overline{g}$.
\end{lem}
\begin{proof} Since 
 $D$ is completely simple, we have $e\,\ar\, g\,\el\, f$ and since
 every $\mathcal{H}$-class in $D$ contains an idempotent, there exists some $h^{2}=h\in D$ such that $h\in L_{e}\cap R_{f}$. We therefore obtain an $E$-square $\begin{bmatrix} e&g\\ h&f\end{bmatrix}$, which by Lemma~\ref{lem:esquarerectband} is  a rectangular band. From Lemma~\ref{lem:singsquaresD} and Corollary~\ref{cor:updown}  we know it must be an up-down singular square, i.e. there exists some idempotent $k$ such that, $ek=e, gk=g, ke=h, \mbox{~and~} kg=f.$ Hence $$\overline{e}~\overline{f}=\overline{e}~\overline{kg}=\overline{e}~\overline{k}~\overline{g}=\overline{ek}~
\overline{g}=\overline{e}~\overline{g}=\overline{eg}=\overline{g}. $$\end{proof}

We now locate a set of generators for $\overline{H}$. 

\begin{lem}\label{lem:generators}
Every element in $\overline{H}$ is a product of elements of form $\overline{e_{11}}~\overline{e_{ij}}~\overline{e_{11}}$ and $(\overline{e_{11}}~\overline{e_{ij}}~\overline{e_{11}})^{-1}$, where 
$j\in [1,n]$ and $i\in I$.
\end{lem}
\begin{proof} By Lemma 1 of \cite{hall:1973}, which itself uses the techniques of
\cite{fitzgerald:1972},  every element of $\overline{H}$  is a product of idempotents of the form $\overline{e_{ij}}$. Let $a\in \overline{H}$. If $a=\overline{e_{ij}}$ then $i=j=1$ and clearly, $a=\overline{e_{11}}~\overline{e_{ij}}~\overline{e_{11}}$. Suppose $a=\overline{e_{i_{1}j_{1}}}\cdots \overline{e_{i_{k}j_{k}}}=\overline{e_{11}}~\overline{e_{i_{1}j_{1}}}\cdots \overline{e_{i_{k}j_{k}}}~\overline{e_{11}}$, where $k\geq 2$. Notice that $e_{11}~\mathcal{R}$$~e_{i_{1}j_{1}}$ in End $F_n(G)$, which implies that $i_{1}=1$. Thus we have
$$
\begin{aligned}
a &=\overline{e_{11}}~\overline{e_{1j_{1}}}~\overline{e_{i_{2}j_{2}}}~\cdots \overline{e_{i_{k}j_{k}}}~\overline{e_{11}}\ \ \ \ \\
                                              & =
\overline{e_{11}}~\overline{e_{1j_{1}}}~(\overline{e_{i_{2}j_{2}}}~\overline{e_{11}}~\overline{e_{1j_{2}}})~\overline{e_{i_{3}j_{3}}}~\cdots \overline{e_{i_{k}j_{k}}}~\overline{e_{11}}\ \ \ \ \  \\
                                              & =
(\overline{e_{11}}~\overline{e_{1j_{1}}}~\overline{e_{i_{2}j_{2}}}~\overline{e_{11}})~\overline{e_{1j_{2}}}~\overline{e_{i_{3}j_{3}}}~\cdots \overline{e_{i_{k}j_{k}}}~\overline{e_{11}}\ \ \ \ \  \\
                                              &=
(\overline{e_{11}}~\overline{e_{1j_{1}}}~\overline{e_{i_{2}j_{2}}}~\overline{e_{11}})\overline{e_{11}}~\overline{e_{1j_{2}}}
~(\overline{e_{i_{3}j_{3}}}~\overline{e_{11}}~\overline{e_{1j_{3}}})~\overline{e_{i_{4}j_{4}}}\cdots \overline{e_{i_{k}j_{k}}}~\overline{e_{11}}  \\
                                              &=
(\overline{e_{11}}~\overline{e_{1j_{1}}}~\overline{e_{i_{2}j_{2}}}~\overline{e_{11}})
(\overline{e_{11}}~\overline{e_{1j_{2}}}~\overline{e_{i_{3}j_{3}}}~\overline{e_{11}})\cdots (\overline{e_{11}}~\overline{e_{1j_{k-1}}}~\overline{e_{i_{k}j_{k}}}~\overline{e_{11}}).
\end{aligned}
$$
Now,  for any $t\in [1,n]$ and $i,j\in I$, we have
$$
\begin{aligned}
(\overline{e_{11}}~\overline{e_{it}}~\overline{e_{11}})(\overline{e_{11}}~\overline{e_{1t}}~\overline{e_{ij}}~\overline{e_{11}})& =\overline{e_{11}}~\overline{e_{it}}~\overline{e_{1t}}~\overline{e_{ij}}~\overline{e_{11}}\ \ \ \ \\
                                              &=
\overline{e_{11}}~\overline{e_{it}}~\overline{e_{ij}}~\overline{e_{11}} \\
                                              &=
\overline{e_{11}}~\overline{e_{ij}}~\overline{e_{11}}.
\end{aligned}
$$
From Proposition~\ref{prop:remarkable} (ii) we have that any product of elements
$\overline{e_{ij}}$ that begins with some $\overline{e_{1j}}$ and ends in some  $\overline{e_{i1}}$ lies in
$\overline{H}$.
We have therefore shown that in $\overline{H}$ we have $\overline{e_{11}}~\overline{e_{1t}}~\overline{e_{ij}}~\overline{e_{11}}=(\overline{e_{11}}~\overline{e_{it}}~\overline{e_{11}})^{-1}
~(\overline{e_{11}}~\overline{e_{ij}}~\overline{e_{11}})$. Hence, every element in $\overline{H}$ is a product of the elements of form $\overline{e_{11}}~\overline{e_{ij}}~\overline{e_{11}}$ and
their inverses.  
\end{proof}

\begin{lem}\label{lem:inverses}
For any $i\in I$ and  $j\in [1,n]$ we have that $\overline{e_{1j}}~\overline{e_{i1}}$ is the inverse of $\overline{e_{11}}~\overline{e_{ij}}~\overline{e_{11}}$.
\end{lem}
\begin{proof} We simply calculate:
$$(\overline{e_{11}}~\overline{e_{ij}}~\overline{e_{11}})(\overline{e_{1j}}~\overline{e_{i1}})=\overline{e_{11}}~
\overline{e_{ij}}~\overline{e_{1j}}~\overline{e_{i1}}=\overline{e_{11}}~\overline{e_{ij}}~\overline{e_{i1}}=\overline{e_{11}}
~\overline{e_{i1}}=\overline{e_{11}}.$$ 
\end{proof}

The next result is immediate from 
Lemmas~\ref{lem:easy} and ~\ref{lem:inverses}. 

\begin{lem}\label{lem:firstrelation} 
If $e_{1j}e_{i1}=e_{11}$, then $\overline{e_{11}}~\overline{e_{ij}}~\overline{e_{11}}=\overline{e_{11}}.$
\end{lem}

\begin{lem}\label{lem:firstequalities} Let $i,l\in I$ and $j,k\in J$.
\begin{enumerate}\item[(i)] If $e_{1j}e_{i1}=e_{1j}e_{l1}$, that is, $p_{ji}=p_{jl}$ in the sandwich matrix $P$, then $\overline{e_{11}}~\overline{e_{ij}}~\overline{e_{11}}=\overline{e_{11}}~\overline{e_{lj}}~\overline{e_{11}}.$
\item[(ii)] If $e_{1j}e_{i1}=e_{1k}e_{i1}$, that is, $p_{ji}=p_{ki}$ in the sandwich matrix $P$, then $\overline{e_{11}}~\overline{e_{ij}}~\overline{e_{11}}=\overline{e_{11}}~\overline{e_{ik}}~\overline{e_{11}}.$
\end{enumerate}
\end{lem}
\begin{proof} (i) Notice that $p_{1i}^{-1}p_{1l}=e_{11}=p_{ji}^{-1}p_{jl}$, so that
from Lemma~\ref{lem:sandwich} we have that the elements of 
$\begin{bmatrix}e_{i1}&e_{ij}\\e_{l1}&e_{lj}\end{bmatrix}$ form a rectangular band.
Thus $e_{ij}=e_{i1}e_{lj}$ and so from Lemma~\ref{lem:easy} we have that 
 $\overline{e_{ij}}=\overline{e_{i1}}\,\overline{e_{lj}}$. So, $\overline{e_{11}}~\overline{e_{ij}}~\overline{e_{11}}=\overline{e_{11}}~\overline{e_{i1}}~\overline{e_{lj}}~\overline{e_{11}}=
\overline{e_{11}}~\overline{e_{lj}}~\overline{e_{11}}.$

(ii) Here we have that  $p_{ji}^{-1}p_{j1}=p_{ki}^{-1}p_{k1}$, 
so that  $\begin{bmatrix} e_{ij}&e_{ik}\\ e_{1j}&e_{1k}\end{bmatrix}$ is a rectangular band and $\overline{e_{ij}}=\overline{e_{ik}}~\overline{e_{1j}}$. So, $\overline{e_{11}}~\overline{e_{ij}}~\overline{e_{11}}=\overline{e_{11}}~\overline{e_{ik}}~\overline{e_{1j}}~\overline{e_{11}}=
\overline{e_{11}}~\overline{e_{ik}}~\overline{e_{11}}.$
\end{proof}

\begin{lem}\label{lem:generalequalities}
For any  $i, i'\in I$, $j,j'\in J$, if  $e_{1j}e_{i1}=e_{1j'}e_{i'1}$, then $\overline{e_{11}}~\overline{e_{ij}}~\overline{e_{11}}=\overline{e_{11}}~\overline{e_{i'j'}}~\overline{e_{11}}.$
\end{lem}
\begin{proof} Let $a=e_{1j}e_{i1}=e_{1j'}e_{i'1}$. By Lemma~\ref{lem:thecolumns}  we can choose a $k\in I$ such that the 
$k$th column of $P$ is $(e_{11}, a,\hdots,a)$. Then $p_{ji}=p_{jk}$ 
and $p_{j'k}=p_{j'i'}$ (this is true even if $j$ or $j'$ is $1$) and our hypothesis now gives that $p_{jk}=p_{j'k}$. The result now follows from three applications of Lemma
~\ref{lem:firstequalities}.
\end{proof}

From now on, we denote $\overline{e_{11}}~\overline{e_{ij}}~\overline{e_{11}}$ with $e_{1j}e_{i1}=a^{-1}$ by $w_{a}$. Notice that $w_{e_{11}}=\overline{e_{11}}$.
Of course, $a=a_g$ for some $g\in G$.

\begin{lem}\label{lem:hom}
With the notation given above, for any $u,v\in H$, we have $w_{u}w_{v}=w_{uv}$ and $w_{u}^{-1}=w_{u^{-1}}$.
\end{lem}
\begin{proof} By Lemma~\ref{lem:thecolumns}, $P$ must contain columns  $(e_{11}, u^{-1}, v^{-1}u^{-1},\cdots)^{T}$ and $(e_{11}, e_{11}, v^{-1},\cdots)^{T}$. For convenience, we suppose that they are the $i$-th and $l$-th columns, respectively. So,  $p_{2i}=e_{12}e_{i1}=u^{-1}$,
$p_{3i}=e_{13}e_{i1}=v^{-1}u^{-1}$, $p_{2l}=e_{12}e_{l1}=e_{11}$ and $p_{3l}=e_{13}e_{l1}=v^{-1}$. It is easy to see that $p_{2i}^{-1}p_{2l}=p_{3i}^{-1}p_{3l}$. Then  $\begin{bmatrix} e_{i2}&e_{i3}\\ e_{l2}&e_{l3}\end{bmatrix}$ is a rectangular band by Lemma~\ref{lem:sandwich}.  
In the notation given above, we have $w_{u}=\overline{e_{11}}~\overline{e_{i2}}~\overline{e_{11}}$, $w_{v}=\overline{e_{11}}~\overline{e_{l3}}~\overline{e_{11}}$ and  $w_{uv}=\overline{e_{11}}~\overline{e_{i3}}~\overline{e_{11}}$. By Lemma~\ref{lem:easy}, $\overline{e_{12}}~\overline{e_{l1}}=\overline{e_{11}}$. We then calculate
$$
\begin{array}{rcl}
w_{u}w_{v}&=&\overline{e_{11}}~\overline{e_{i2}}~\overline{e_{11}}~\overline{e_{11}}~\overline{e_{l3}}~\overline{e_{11}}\\
                                              & =&
\overline{e_{11}}~\overline{e_{i2}}~\overline{e_{12}}~\overline{e_{l1}}~\overline{e_{l3}}~\overline{e_{11}} \\
                                              & =&
\overline{e_{11}}~\overline{e_{i2}}~\overline{e_{l3}}~\overline{e_{11}} \\
                                              &=&
\overline{e_{11}}~\overline{e_{i3}}~\overline{e_{11}} \ \ \ \ (\mbox{since } \begin{bmatrix} e_{i2}&e_{i3}\\ e_{l2}&e_{l3}\end{bmatrix} \mbox{is a rectangular band} )\\
&=&
w_{uv}.
\end{array}
$$

Finally, we show $w_{u}^{-1}=w_{u^{-1}}$. This follows since $\overline{e_{11}}=w_{{e_{11}}}=w_{u^{-1}u}=w_{u^{-1}}w_{u}$. \end{proof}

It follows from Lemma~\ref{lem:generators} that any element of 
$\overline{H}$ can be expressed as some $w_{a}$ for some  $a\in H_{11}$.

\begin{thm}\label{thm:theresult}
Let $F_n(G)=\bigcup_{i=1}^n Gx_{i}$ be a finite rank $n$ free (left) group act with 
$n\geq 3$, and let $\en F_n(G)$ the endomorphism monoid of $F_n(G)$. Let $e$ be an arbitrary rank 1 idempotent. Then the maximal subgroup $\overline{H}$ of $\ig(E)$ containing $\overline{e}$ is isomorphic to $G$. In other words, every group arises as the maximal subgroup of the free idempotent generated semigroup arising from the endomorphism monoid of a finite dimensional free group act.
\end{thm}
\begin{proof}
Define a mapping $\phi: \overline{H}\longrightarrow H$ by $w_{a}\mapsto a$. By Lemma~\ref{lem:thecolumns}, $\phi$ is onto and it is a morphism by Lemma~\ref{lem:hom}. To prove that $\phi$ is one-one, suppose $w_{a}\phi=e_{11}$. Then by definition we have that  $a=e_{11}$ and we have observed that $w_{e_{11}}=\overline{e_{11}}$. Thus $\overline{H}$ is isomorphic to $H$ and we have observed in Lemma~\ref{lem:Goccurs} that $H$ is isomorphic to $G$. Indeed, the map $w_{a_g}\mapsto g$ is an isomorphism.
\end{proof}

We remark that if $G$ is finite, then clearly so is  $E(\en F_n(G))$. However, in
\cite{gray:2012} it is proven that if $G$ is {\em finitely presented}, then
$G$ is a maximal subgroup of $\ig(E)$ for some {\em finite} $E$: our construction makes no headway in this direction.

\end{document}